\documentclass[12pt]{article}

\usepackage{aeguill}
\usepackage[francais,english]{babel}
\usepackage[utf8]{inputenc}
\usepackage[T1]{fontenc}
\usepackage{amsfonts}
\usepackage{amssymb}
\usepackage{xspace}
\usepackage{graphicx}
\usepackage{amsmath}
\usepackage{amsthm}
\usepackage{stmaryrd}
\usepackage{endnotes}
\usepackage{yfonts}
\usepackage{makeidx}
\usepackage{fancyhdr}
\usepackage{pgf}
\usepackage{dsfont}
\usepackage{bbm}
\usepackage[a4paper]{geometry}
\usepackage{pdflscape}
\usepackage{pdfpages}
\usepackage{hyperref}

\theoremstyle{plain}
\newtheorem{thm}{Theorem}[section]
\newtheorem{prop}[thm]{Proposition}
\newtheorem{lem}[thm]{Lemma}
\newtheorem{cor}[thm]{Corollary}

\theoremstyle{definition}
\newtheorem{df}[thm]{Definition}

\newcommand{\N}{\mathbb{N}}

\newcommand{\R}{\mathbb{R}}

\newcommand{\dd}[1]{\ensuremath{\operatorname{d}\!{#1}}}
\newcommand{\sui}{_{n\geq 1}}
\newcommand{\fcar}{\mathbbm{1}}
\newcommand{\n[1]}{||#1||}
\newcommand{\tend}{\rightarrow}

\newcommand{\ab[1]}{|#1|}
\newcommand{\cont[1]}{\mathcal{C}^0(#1)}

\newcommand{\nti}{n\tend\infty}

\newcommand{\Gad}{\mathrm{GL}(d,\R)\ltimes \R^d}
\newcommand{\Gsad}{\mathrm{SL}(d,\R)\ltimes \R^d}

\newcommand{\Pj[1]}{\mathbb{P}(#1)}

\newcommand{\Gl}{\mathrm{GL}(V)}

\newcommand{\ma}{random walk\xspace}

\newcommand{\Lie[1]}{\mathfrak{#1}}

\newcommand{\sm}{\sigma_{\mu}}

\newcommand{\Xkd}{X_{k,\,d}}
\newcommand{\lun}{\lambda_1}
\newcommand{\lpun}{\lambda^{\prime}_1}
\newcommand{\Gr}{\mathrm{Gr}}

\DeclareMathOperator{\Ker}{Ker}
\DeclareMathOperator{\Ima}{Im}

\DeclareMathOperator{\Supp}{Supp}

\numberwithin{equation}{section}
\addto{\captionsenglish}{}
\addto{\captionsenglish}{}
\def\bq{\begin{quotation}}
\def\eq{\end{quotation}}

\begin{document}

\title{Recurrence on Affine Grassmannians}
\author{Yves Benoist and Caroline Bru\`ere}
\date{}
\maketitle
\selectlanguage{english}

\begin{abstract} 
We study the action of the affine group $G$ of $\R^d$
on the space $X_{k,\,d}$ of $k$-dimensional affine subspaces. 
Given a compactly-supported 
Zariski dense probability measure $\mu$ on $G$, 
we show that $X_{k,\,d}$ supports a $\mu$-stationary measure $\nu$
if and only if the $(k\!+\!1)^{\rm th}$-Lyapunov exponent
of $\mu$ is strictly negative. 
In particular, when $\mu$ is symmetric, $\nu$ exists 
if and only if $2k\geq d$. 
\end{abstract}

{\footnotesize
\tableofcontents
}
\newpage

%1
\section{Introduction}

%11
\subsection{Recurrence and Lyapunov Exponents}
\hspace*{1em}\let\thefootnote\relax\footnote{\hspace*{-2em}
Key words: Affine group, Grassmannian,
random walk, recurrence, stationary probability.\\
AMS-MSC : 22E40, 60J20.}
Consider a locally compact group $G$
acting continuously on a locally compact second countable space $X$
and $\mu$ a probability measure on $G$. 
The \emph{associated random walk on $X$}
is the Markov chain over $X$ 
defined by the transition probabilities 
$P_x=\mu *\delta_x$ for all $x\in X$. 
Our aim is to study the recurrence properties of such a random walk.
We will not focus here on the \emph{almost sure recurrence} 
as in \cite{BouBabEl} and \cite{Bru1}
but on the \emph{recurrence in law}
as in \cite{Art}, \cite{BouPic} and\cite{EskMarg}.

\begin{df} \label{defrectrans}
The random walk on $X$ is \emph{recurrent in law} at a point
 $x\in X$ 
if for all $\varepsilon>0$, there exists a compact set $C\subset X$
and $n_0\in\N^*$ such that for all $n \geq n_0$: 
\[\mu^{*n}*\delta_x(C)\geq 1-\varepsilon.\] 
The random walk on $X$ is \emph{uniformly recurrent in law} 
if the same compact set $C$ can be chosen for all the starting points $x$.
A probability measure $\nu$ on $X$ is said to be 
\emph{$\mu$-stationary} or \emph{$\mu$-invariant}
if one has $\mu * \nu= \nu. $
\end{df}

Those definitions are tightly linked. Indeed,
there exists a $\mu$-stationary probability measure on $X$
if and only if the random walk on $X$ is recurrent in law 
at some point $x\in X$
(see Lemma  \ref{urloimesinv} for one implication).

In this paper, $G$ will always be a real algebraic group
acting algebraically on a real algebraic variety $X$; 
the measure $\mu$ will be compactly supported and
\emph{Zariski dense}, which means that its support spans a Zariski dense subgroup 
in $G$.  

When $G$ is a reductive group and $X=G/H$ is an algebraic homogeneous space, 
it is proven in \cite{Art} that there exists a $\mu$-stationary probability measure on $X$ 
if and only if $X$ is compact. 
The aim of our article is to focus on situations 
where the algebraic group $G$ is not reductive. 
In particular, in Corollary \ref{casSL}, we will exhibit 
examples of non-compact homogeneous spaces 
on which there always exists a $\mu$-stationary
probability measure.

The key tool in our analysis will be to link 
the recurrence properties of these random walks to the Lyapunov exponents of $\mu$.
The definition of these Lyapunov exponents 
depends on the choice of a linear action
of $G$ on $\R^d$.

\begin{df}\label{deflyap} Given  
a linear action of $G$ on $\R^d$,
the \emph{Lyapunov exponents} of $\mu$ 
are the real numbers $\lambda_1,\,\hdots,\, \lambda_d$
such that, 
for all $1\leq p \leq d$, we have
\begin{equation}
\label{eqndeflya}
\lambda_1+\hdots+\lambda_p=
\lim_{\nti} \frac{1}{n} \int_G \log \n[\Lambda^p\, g] \dd\mu^{*n} (g).
\end{equation}
\end{df} 
The sequence of Lyapunov exponents is always decreasing:
$\lambda_1\geq \hdots \geq \lambda_d
$ (see \cite[Prop 1.2]{Led84}). 
More properties of these exponents are given in  \cite{Ose68}, \cite{Led84};
their use in the context of reductive groups is detailed in  
\cite{Furst63}, \cite{GuiRaug85}, \cite{GoldMarg} and \cite{Livre}.

%12
\subsection{Action on the Affine Grassmannians}
We assume now that $G$ is either the affine group $G=\Gad$ or 
the special affine group $G=\Gsad$.
For $1\leq p\leq d$, 
we denote by $\lambda_{p}$ 
the $p^{\rm th}$-Lyapunov exponent   
corresponding to the linear action of $G$ on $\R^d$.
For instance,
in dimension $d=1$, one has
\[\textstyle\lambda_1 = \int_{\R^*\ltimes\R} \log \ab[a] \dd \mu(a,u)\]
where 
$g=(a,u)\in \R^*\ltimes\R$. 
For any $d\geq 1$, Bougerol and Picard have shown in \cite{BouPic}
that				
there exists a $\mu$-stationary probability measure on $\R^d$ 
if and only if the first Lyapunov exponent of $\mu$ is strictly negative:
$\lambda_1<0$.

The main result of this paper is the following Theorem \ref{recgrassaff}, 
which extends this equivalence 
to the affine Grassmannians $\Xkd$ where $0\leq k < d$.
By definition the affine Grassmannian $\Xkd$ is the space of
$k$-dimensional affine subspaces of $\R^d$. The group $G$ acts transitively on
$\Xkd$.

\begin{thm}\label{recgrassaff}
Let $G$ be the affine group or the special affine group of $\R^d$,  let
$\mu$ be a Zariski dense probability measure with compact support on $G$
and let $0\leq k<d$.\\ 
a) If $\lambda_{k+1}\geq 0$, then the random walk on $\Xkd$ 
is nowhere recurrent in law, 
there exists no $\mu$-stationary probability measure on $\Xkd$, and
for all $x$ in $\Xkd$ the sequence of means of transition probabilities
weakly converges to $0$:
\[\textstyle\frac{1}{n}\sum_{j=1}^n \mu^{*j}* \delta_x\xrightarrow[\nti]{} 0.\]
b) If $\lambda_{k+1}<0$, then the random walk on $\Xkd$ 
is uniformly recurrent in law, 
there exists a unique $\mu$-stationary probability measure $\nu$ on  $\Xkd$,  
and for all $x$ in $\Xkd$ the sequence of means of transition probabilities
weakly converges to $\nu$:
\[\textstyle\frac{1}{n}\sum_{j=1}^n \mu^{*j}* \delta_x\xrightarrow[\nti]{} \nu.\]
\end{thm}
The result of Bougerol and Picard in \cite{BouPic}
covers the $k=0$ case. 
In fact, their proof uses only the weaker assumption
that $\mu$ has a finite first moment and 
that its support does not preserve any proper affine
subspace of $\R^d$.

The following Corollary, which is particularly noteworthy insofar 
as it does not mention Lyapunov exponents, 
is deduced from Theorem \ref{recgrassaff}. 

\begin{cor}\label{cassymetrique} Assume $\mu$ is symmetric. 
Then there exists a $\mu$-stationary probability measure
$\nu$ on $\Xkd$ if and only if
$2k\geq d$. 
In this case, $\nu$ is unique.
\end{cor}

\begin{proof}
Since $\mu$ is symmetric, for all $1\leq p\leq d$, the Lyapunov exponents 
satisfy the equalities 
$\lambda_p=-\lambda_{d+1-p}$. Moreover, 
since $\mu$ is Zariski dense in $G$, 
it follows from the Guivarc'h-Raugi simplicity theorem that
the sequence of Lyapunov exponents is strictly decreasing:
$\lambda_1>\cdots >\lambda_d$ (see \cite[Corol. 10.15]{Livre}).
Therefore one has the equivalence 
$\lambda_{k+1}<0 \Longleftrightarrow 2k\geq d$. 
\end{proof}

\begin{cor}
\label{casSL} 
Let $d\geq 2$. 
When $G$ is the special affine group and  $k=d-1$, 
there exists a unique $\mu$-stationary probability measure on $\Xkd$. 
\end{cor}

\begin{proof}
In this case, the sum of the Lyapunov exponents is zero.
Hence, the simpli\-city of the Lyapunov exponents
implies
$\lambda_d<0$. 
\end{proof}

For instance, when $G$ is the special affine group of $\R^2$, 
the random walk 
on the space of affine lines of $\R^2$
is always uniformly recurrent in law while  
the random walk on the space of points of $\R^2$ is nowhere recurrent in law.

%13
\subsection{Action on $X_{V,W}$}

By embedding the affine Grassmannian $\Xkd$ of $\R^d$ in 
the projective space of a suitable exterior power $V$ of $\R^{d+1}$,
we will deduce 
Theorem \ref{recgrassaff} from the following Theorem \ref{thmeq}:

We first need two definitions.
An algebraic group $G$ is \emph{Zariski connected} 
if it is connected for the Zariski topology.
A linear action of $G$ on a vector space $W$ 
is \emph{proximal} 
if there exists a rank $1$ linear endomorphism $\pi$ of $W$
which is a limit of a sequence $\lambda_n\gamma_n$ with 
$\lambda_n>0$ and $\gamma_n$ in $\Gamma$.

\begin{thm}\label{thmeq} Let $V$ be a finite-dimensional real vector space,
$G$ a Zariski connected algebraic subgroup of $\Gl$, 
$W$ a $G$-invariant subspace of $V$ such that\\
$(H1)$ $G$ acts irreducibly and proximally on $W$ and on $W':=V/W$.\\ 
$(H2)$ The representations of $G$ in $W$ and $W'$ 
are not equivalent.\\ 
$(H3)$ $W$ has no $G$-invariant complementary subspace in $V$.\\
Let $X_{V,W}:=\Pj[V]\smallsetminus\Pj[W]$, let $\mu$ be a Zariski dense probability measure with compact support on $G$ and
let $\lun=\lambda_{1,W}$ and $\lpun=\lambda_{1,W'}$ be
the first Lyapunov exponents of $\mu$ in $W$ and $W'$ respectively. \\
a) If $\lun\geq \lpun$, then the \ma on $X_{V,W}$ 
is nowhere recurrent in law, 
there exists no $\mu$-stationary probability measure on $X_{V,W}$,
and for all $x$ in $X_{V,W}$ one has the weak convergence
$\textstyle\frac{1}{n}\sum_{j=1}^n \mu^{*j}* \delta_x\xrightarrow[\nti]{} 0.$\\
b) If $\lun<\lpun$, 
then the \ma on $X_{V,W}$ is uniformly recurrent in law,
there exists a unique 
$\mu$-stationary probability measure $\nu$ on  $X_{V,W}$,  
and for all $x$ in $X_{V,W}$, one has the weak convergence
$\textstyle\frac{1}{n}\sum_{j=1}^n \mu^{*j}* \delta_x\xrightarrow[\nti]{} \nu.$
\end{thm}

%14
\subsection{Strategy of the Proof}

In Chapter \ref{passage}, we explain how to embed the affine Grassmannian
$X_{k,d}$ in the variety 
$\Pj[\Lambda^{k+1}\R^{d+1}]\smallsetminus\Pj[\Lambda^{k+1}\R^{d}]$ 
and we deduce
Theorem \ref{recgrassaff} from Theorem \ref{thmeq}.

The last three
chapters 
will deal with the proof of Theorem \ref{thmeq}.
 
In Chapter \ref{thmeqrecloi}, we prove the uniform recurrence in law when $\lun < \lpun$
(Corollary \ref{l1lp1urloi}). The crux of the proof is the construction of a proper function 
on $X_{V,W}$ which is contracted by the averaging operator
(Proposition \ref{PmuUCH}). 

In Chapter \ref{l1lp1nrecloi}, 
we prove the non-recurrence in law when $\lun\geq \lpun$ (Proposition \ref{l1lp1nonrecloiprop}). 
The key point is the study of the ratio of the norms in $W$ and in $W'$  of a random product $b_1\cdots b_n$. On the one hand, the existence of a $\mu$-stationary probability measure
on $X_{V,W}$ would imply that these ratios are bounded
(Lemma \ref{liminfdnanfinie}). 
On the other hand, when $\lun\geq \lpun$, the Law of Large Numbers and the Law of Iterated Logarithms for these products
prevent these ratios from being bounded
(Lemma \ref{liminfdnaninfinie}).

In Chapter \ref{thmequnicite}, we prove
the uniqueness of the $\mu$-stationary measure on $X_{V,W}$
(Proposition \ref{unicitedirac}).
Indeed, using the joining measure (Corollary \ref{corjoista}) of two distinct $\mu$-stationary probability 
measures on $X_{V,W}$, 
we construct (Lemma \ref{lemyvw}) a $\mu$-stationary measure 
$\overline{\nu}$ on the space $\Pj[W\oplus W']\smallsetminus (\Pj[W]\cup\Pj[W'])$.
This contradicts the classification of stationary  measures in \cite{Art}
since this space does not contain compact $G$-orbits
(Lemma \ref{pasdorbitescompactesdansX}). 
The weak convergence of the sequence of means of transition probabilities
follows easily (Corollary  \ref{corthmeqlunlpun}). 

In Appendix \ref{seclimlaw}, we collect known facts on random walks on reductive groups.\vspace{1em}

In this paper, all the vector spaces will be 
finite dimensional real vector spaces,
all the measures will be Borel measures and we will not distinguish between a real algebraic group
and its group of real points.

%20
\section{Recurrence on affine Grassmannians}\label{passage}   
\bq
We explain first how to deduce Theorem \ref{thmeq}
from Theorem \ref{recgrassaff}
\eq

We use the notation of Theorem \ref{thmeq}.
The group $G$ is the affine group or the special affine group of $\R^d$, 
the space $\Xkd$ is the affine Grassmannian of $\R^d$,
the probability measure $\mu$ on $G$ is Zariski dense and compactly supported.

Let us construct $G$-vector spaces $W\subset V$ 
to which we will apply Theorem  \ref{thmeq}.  
We identify the affine space $\R^d$ with 
the affine hyperplane of $\R^{d+1}=\R^d\oplus\R$:
$$
\mathcal{A}=\{(w,\,1)\,|\,w\in\R^d\}.
$$
The group $G$ is then a subgroup of
$\mathrm{GL}(d+1,\,\R)$, which stabilizes $\mathcal{A}$,
and we have
$$
X_{k,d}=\Gr_{k+1}(d+1)\smallsetminus\Gr_{k+1}(d),
$$
where $\Gr_{k+1}(d\! +\! 1)$ and $\Gr_{k+1}(d)$ are  
the Grassmannians of $(k\! +\! 1)$-dimensional vector subspaces of 
$\R^{d+1}$ and of $\R^d$ respectively. 
Now, let 
$$
V:=\Lambda^{k+1}\R^{d+1}
\;\;{\rm and} \;\;
W:=\Lambda^{k+1}\R^{d}.
$$ 
The group $G$ acts linearly on the vector space $V$ 
and leaves invariant its vector subspace $W$.
The Pl\"{u}cker map
$$
\varphi\; :\;\Gr_{k+1}(d+1)\longrightarrow \Pj[V]
\;\; ;\;\;
U\longmapsto \Lambda^{k+1}U.
$$
is an embedding of the Grassmannian variety in the projective space of $V$.
It induces a $G$-equivariant injection
\[\varphi:\Xkd\hookrightarrow X_{V,W}:= \Pj[V]\smallsetminus\Pj[W].\]

\begin{prop}\label{constructionVWWp} With the above notations,\\
a) Hypotheses $(H1)$, $(H2)$, $(H3)$ hold for these $V$, $W$ and $W'=V/W$.\\
b) The $G$-equivariant inclusion 
$\Xkd\hookrightarrow X_{V,W}$ has closed image.\\
c) We have the equality $\lambda_{k+1}=\lambda_{1,\,W}-\lambda_{1,\,W'}$.
\end{prop}

\begin{proof}[Proof of Proposition \ref{constructionVWWp}]
a) $(H1)$ : The representation of ${\rm SL}(d,\R)$ in $W=\Lambda^{k+1}\R^d$
is irreducible by \cite[Chap. 8.13.1.4]{BourGALie78}.
This representation is proximal since the image in ${\rm GL}(W)$ 
of a diagonal element of $G$ 
with positive distinct eigenvalues is a proximal element of 
${\rm GL}(W)$. 
The same is true for the representation in $W'\simeq\Lambda^k\R^d$.\\
$(H2)$ : The fact that the representations of ${\rm SL}(d,\R)$ in $W$ and $W'$ 
are not equivalent is also proven in \cite[Chap. 8.13.1.4]{BourGALie78}.\\
$(H3)$ : Since the representation of ${\rm SL}(d,\R)$ in $W$ and $W'$ are 
irreducible and are not equivalent,
by Schur's lemma, the only 
${\rm SL}(d,\R)$-invariant complementary subspace of $W$ in 
$V\simeq W\oplus W'$
is $W'$. But $W'$ is not invariant by the translations of $G$. 
 
b) The image $\varphi(X_{k,d})$ is closed in $X_{V,W}$ since
$\varphi^{-1}(\Pj[W])=\Gr_{k+1}(d)$.

c) This equality is the difference of the equalities
\[
\lambda_{1,\,W}=\lambda_1+\hdots+\lambda_{k+1}
\;\;\; {\rm and}\;\;\;
\lambda_{1,\,W'}=\lambda_1+\hdots+\lambda_{k}
\]
which follow from the very definition \eqref{eqndeflya} 
of the Lyapunov exponents.
\end{proof}

\begin{proof}[Proof of Theorem \ref{thmeq} $\Longrightarrow$ Theorem \ref{recgrassaff}]
We use  Proposition \ref{constructionVWWp}. 

If $\lambda_{k+1}\geq 0$, 
then we can apply Theorem \ref{thmeq}
in the case where  $\lambda_{1,\,W}\geq\lambda_{1,\,W'}$,
and there can be no $\mu$-stationary probability measure on $\Xkd$. 

Conversely, if $\lambda_{k+1}<0$, we are in the case where
 $\lambda_{1,\,W}<\lambda_{1,\,W'}$. 
Since $\Xkd$ is a $G$-invariant closed subset of $X_{V,W}$, 
we obtain uniform recurrence in law on $\Xkd$.
Lemma \ref{urloimesinv} then ensures
the existence of a $\mu$-stationary probability measure on $\Xkd$, 
which is thus the unique 
$\mu$-stationary probability measure on $X_{V,W}$. 
\end{proof}

%30
\section{Uniform Recurrence When $\lambda_1<\lambda_1'$}\label{thmeqrecloi}
\bq
The goal of this Chapter is to show that the random walk on $X_{V,W}$
is uniformly recurrent in law
when $\lambda_1<\lambda_1'$ (Corollary \ref{l1lp1urloi}). 
\eq

%31
\subsection{The Contraction Hypothesis}
\bq
We recall in this section the uniform contraction hypothesis
and why this condition implies the uniform recurrence in law.
\eq

The setting is very general (see \cite{MeyTwe}, \cite{EskMarg} or \cite{BenQuiFinVol}
for more details).
Let $X$ be a locally compact second-countable space
and $P$ a Markov-Feller operator on $X$.

\begin{df}\label{defUCH}
The operator $P$ satisfies the \emph{uniform contraction hypothesis} 
(UCH)
if there exists a 
proper map $u:X\rightarrow [0,\,\infty[$ 
and two constants $0<a<1$ and $b>0$ such that, over $X$,  
\begin{equation}
\label{eqnpulaub}
Pu\leq au+b.
\end{equation}
\end{df}

We recall that a map is \emph{proper} if the inverse image 
of every compact set is relatively compact.  
The definition of recurrence in law extends to 
Markov chains on $X$.
Uniform recurrence in law is fundamentally linked with (UCH):

\begin{prop}\label{UCHurloi} If $P$ satisfies (UCH), 
then the associated Markov chain on $X$ is uniformly recurrent in law. 
\end{prop}

\begin{proof}
See \cite[Thm 15.0.1]{MeyTwe}, \cite[Lem 3.1]{EskMarg} or \cite[Lem 2.1]{BenQuiFinVol}.
\end{proof}

\begin{lem}
\label{urloimesinv} 
If $P$ is recurrent in law at point $x\in X$, 
there exists a $P$-invariant probability measure on $X$.  
\end{lem}

\begin{proof} 
By the Banach-Alaoglu Theorem, 
the sequence of means of transition probabilities $\nu_n=\frac{1}{n}\sum_{j=1}^n P^j_x$ 
has at least one accumulation point $\nu_\infty$
for the weak-$*$ topology.
This finite measure $\nu_{\infty}$ is $P$-invariant.
Since $P$ is recurrent in law at $x$, there is no escape of mass 
and $\nu_{\infty}$ is a probability measure. 
\end{proof}

The following lemma is a useful tool to check (UCH).

\begin{lem} 
\label{lempnouch}
Let $n\geq 1$. If $P^n$ satisfies $(UCH)$ 
then $P$ satisfies $(UCH)$ too.
\end{lem} 

\begin{proof}
Let $u$ be the proper map and $a$,$b$ the constants such that
$P^nu\leq au+b$ over $X$.
Let $\alpha_k=a^{-k/n}$ for $0\leq k\leq n-1$, 
$a'=a^{1/n}$,
$b'=\frac{a'}{a} b$. 
Then the proper map $u':X\rightarrow \R_+$ 
defined by $u'=\sum_{k=0}^{n-1} \alpha_k P^k u$
satisfies the inequality
$Pu'\leq a'u'+b'$
on $X$,
and thus $P$ satisfies(UCH). 
\end{proof}

%32
\subsection{Finding a Contracted Function}
\bq
In this section, we use again 
the notations and assumptions of Theorem \ref{thmeq}. 
We will prove that the averaging operator
satisfies the uniform contraction hypothesis. 
\eq

We recall that $W\subset V$ are real vector spaces,
$G$ is a Zariski connected algebraic subgroup of $\Gl$
preserving $W$ and satisfying $(H1)$, $(H2)$, $(H3)$.
We identify the quotient $W'=V/W$ with a complementary subspace $W_s$ of $W$
in $V$.
Note that this subspace $W_s$ is not $G$-invariant.
We recall also that $\mu$ is a Zariski dense probability measure on $G$
with compact support and that $\lun$ and $\lpun$ 
are the first Lyapunov exponents of $\mu$ in $W$ and $W'$,
and that we are studying the associated random walk on
the $G$-space $X_{V,W}:=\Pj[V]\smallsetminus\Pj[W]$.

The corresponding Markov operator 
$P_{\mu}:\cont[X_{V,W}]\longrightarrow\cont[X_{V,W}]$ is given by 
$$
\textstyle
P_{\mu}f (x)=\int_G f(gx) \dd\mu(g).
$$

\begin{prop}\label{PmuUCH} Same notations and assumptions
as in Theorem \ref{thmeq}.\\
If $\lun<\lpun$, then the Markov operator $P_{\mu}$ satisfies (UCH). 
\end{prop}

\begin{proof} 
The space $X_{V,W}$
can be seen as the set 
$$
X_{V,W}=\{[w,\, w']\,|\, w\in W,\, w'\in W_s\smallsetminus\{0\}\}.
$$ 
Choose a norm on $V$, and, for $\delta>0$, consider the functions
\[
u_{\delta}\; :\;
X_{V,W}\longrightarrow\R_+
\;\; ;\;\; [w,\,w']\longmapsto \tfrac{\n[w]^{\delta}}{\n[w']^{\delta}}.\]
These functions are proper and well-defined. 
We want to find
$\delta>0$, $ a\in ]0,\,1[$, $b>0$, $n_0\in\N^*$ such that, 
over $X_{V,W}$, one has the inequality
\begin{equation}
\label{ineqUCH}
P_{\mu}^{n_0}u_{\delta}\leq a u_{\delta} +b\, .
\end{equation}

Since $W$ is $G$-invariant, we can write $g\in G$ as
\begin{equation}
\label{eqnagcgdg}
g=\begin{pmatrix}
a_g & c_g \\
0 & d_g
\end{pmatrix}\;
\mbox{\rm with $a_g\in \mathrm{GL}(W),\,d_g\in \mathrm{GL}(W_s), 
\, c_g\in \mathcal{L}(W_s,\,W)$.}
\end{equation}
Let $0<\varepsilon < \frac{\lpun-\lun}{8}$.
Then, by a lemma due to Furstenberg (cf. \cite[Thm 4.28]{Livre}, \cite{Furst63})
since $G$ 
acts irreducibly on $W$ and $W'$ 
there exists $n_0\in\N^*$ such that 
for all $n\geq n_0$,
for all non-zero $w\in W,\, w'\in W_s$,
the following inequalities hold:
\begin{align}\label{inegFurstlun}
\lun-\varepsilon &\leq \frac{1}{n}\int_G \log \frac{\n[a_g w]}{\n[w]}\dd \mu^{*n}(g)
\leq \lun+\varepsilon, \\ \label{inegFurstlpun}
\lpun-\varepsilon &\leq \frac{1}{n}\int_G \log \frac{\n[d_g w']}{\n[w']}\dd \mu^{*n}(g)
\leq \lpun+\varepsilon.
\end{align}

For $\delta >0$ and $x=[w,\,w']\in X_{V,W}$, one computes
\[P^{n_0}_{\mu} u_{\delta}(x)=u_{1}(x)^{\delta}  \;
\int_G\frac{u_1(gx)^{\delta}}{u_1(x)^{\delta}}
\dd\mu^{*n_0}(g).\]
We will give an upper bound for the right-hand integral
for all $x$ in the complementary set
of some compact $K$ in $X$ ; 
since map $P_{\mu}^{n_0}u_{\delta}$ is bounded on the compact set $K$, 
this will give inequality (\ref{ineqUCH}). 
Let $c>0$ be the constant defined by
\[c^{-1}= \frac{4}{n_0(\lpun-\lun)}\int_G\n[c_g]\,\n[a_g^{-1}] \dd \mu^{*n_0}(g).\]  
Let $K$ be the compact subset of $X_{V,W}$ given by
\[K=\{[w,\,w']\,|\, w\in W, \; w'\in W_s,\; \n[w']\geq c\, \n[w] \}.\] 
For $\mu^{*n_0}$-almost every $g\in G$,
for all $x\in X\smallsetminus K$,
the following ratio 
is bounded:
\[
\frac{u_1(gx)}{u_1(x)} 
\; =\; 
\frac{\n[a_gw+c_gw']}{\n[w]}\, \frac{\n[w']}{\n[d_gw']}
\;\leq\; 
\sup_{g\in \Supp \mu^{*n_0}}
\n[d_g^{-1}](\n[a_g]+c\n[c_g]).
\]
Therefore, we can find some constant $M_{n_0}>0$ 
such that for all $\delta>0$, 
for all $x\in X\smallsetminus K$, 
for $\mu^{*n_0}$-almost every $g\in G$,
we can write
\[ \frac{u_{1}(gx)^{\delta}}{u_{1}(x)^{\delta}}=e^{\delta\log \frac{u_1(gx)}{u_1(x)}}
\leq 1+\delta \log  \frac{u_1(gx)}{u_1(x)} +\delta^2 M_{n_0}.\]
For all $x\in X\smallsetminus K$, for $\mu^{*n_0}$-almost every $g\in G$, 
the following upper bound holds: 
\begin{align*}
\log \frac{u_1(gx)}{u_1(x)} &= \log\frac{\n[a_g w]}{\n[w]} 
- \log\frac{\n[d_g w']}{\n[w']}+ \log \frac{\n[a_gw+c_gw']}{\n[a_gw]}\\
&\leq \log\frac{\n[a_g w]}{\n[w]} 
- \log\frac{\n[d_g w']}{\n[w']}+\frac{\n[c_gw']}{\n[a_gw]} \\
&\leq \log\frac{\n[a_g w]}{\n[w]} 
- \log\frac{\n[d_g w']}{\n[w']}+ \n[c_g]\n[a_g^{-1}]c . 
\end{align*}
Using inequalities (\ref{inegFurstlun}), (\ref{inegFurstlpun})
and the definition of $c$, we get the inequality
\begin{equation*}
\int_G \log \frac{u_1(gx)}{u_1(x)} \dd \mu^{*n_0}(g)
\;\leq\; 
n_0(\lun-\lpun+2\varepsilon) + \frac{n_0(\lpun-\lun)}{4} 
\;\leq\; 
\frac{n_0(\lun-\lpun)}{2}. 
\end{equation*}
Let $\kappa=\frac{n_0(\lpun-\lun)}{2} >0$. We then get the upper bound,
for all $x\in X\smallsetminus K$,
\[\int_G \frac{u_{1}(gx)^{\delta}}{u_{1}(x)^{\delta}} \dd \mu^{*n_0}(g)\leq 1-\delta \kappa 
+\delta^2 M_{n_0}.\]

Choose $\delta >0$ such that 
$a_{n_0,\,\delta}:= 1-\delta \kappa +\delta^2 M_{n_0}$ 
is strictly between $0$ and $1$.
Therefore, since $K$ is compact, there exists a constant $b_{n_0,\,\delta}$ such that 
for all $x\in X$:
\[P^{n_0}_{\mu}u_{\delta} (x) \leq a_{n_0,\,\delta}\, u_{\delta}(x) +b_{n_0,\, \delta},\]
and, by Lemma \ref{lempnouch}, the operator $P_{\mu}$ satisfies (UCH). 
\end{proof}

\begin{cor}\label{l1lp1urloi} Same notations and assumptions 
as in Theorem \ref{thmeq}.\\
If $\lun<\lpun$, 
then the \ma on $X$ is uniformly recurrent in law.
\end{cor}

\begin{proof}
This is a direct consequence of Proposition \ref{PmuUCH}: 
since $P_{\mu}$ satisfies (UCH), 
we only need to apply Proposition \ref{UCHurloi}. 
\end{proof}

%40
\section{ Non-Recurrence in Law When $\lambda_1 \geq  \lambda_1'$} 
\label{l1lp1nrecloi}  

\bq
The goal of this Chapter is to show that the random walk on $X_{V,W}$
is nowhere recurrent in law
when $\lambda_1\geq\lambda_1'$ (Proposition \ref{l1lp1nonrecloiprop}). 
\eq

%41
\subsection{The Limit Measures}
\bq
We recall in this section the definition and the properties of the
limit probability measures associated to a stationary measure.
\eq

The setting is very general.
Let $G$ be a  locally compact group acting on a 
second countable locally compact space $X$ and 
$\mu$ be a probability measure on $G$.
Let $B$ be the product space $B=G^{\N^*}$ and
$\beta$ be the product measure $\beta=\mu^{\otimes\N^*}$. 
%and $S$ be the shift map on $B$ 
%$$S:B\longrightarrow B\;\; ;\;\;
%b=(b_1,b_2,\,\hdots)\longmapsto Sb=(b_2,b_3,\,\hdots)\; .$$
The following lemma is due to Furstenberg. See \cite[Lem 3.2]{StatI} 
or \cite[Lemma 2.17]{Livre}.

\begin{lem}
\label{nub}
Let $\nu$ be a $\mu$-stationary probability measure on $X$. 
For $\beta$-almost every $b\in B$, the sequence
$(b_1\cdots b_n)_*\nu$
of probability measures on $X$ has a limit $\nu_b$, 
which we will call \emph{limit probability}.  
%Furthermore, for $\beta$-almost every $b\in B$, 
%the following equivariance property holds:
%\begin{equation}\label{equivariancenub}
%\nu_b= {b_1}_*\nu_{Sb}.
%\end{equation}
Moreover, we have
$\nu=\int_{B}\nu_b\dd\beta(b).$
\end{lem}
The following construction will be useful 
in Chapter \ref{secunista}. See \cite[Cor 3.5]{StatI} for a proof.

\begin{cor} 
\label{corjoista}
Let $\nu_1$ and $\nu_2$ be two 
$\mu$-stationary probability measures on $X$. Then
the probability measure on $X\times X$
\begin{equation}
\label{eqnjoista}
\nu_1\boxtimes\nu_2 := \int_B \nu_{1,\,b}\otimes\nu_{2,\,b}\dd \beta(b)
\end{equation}
is $\mu$-stationary. It is called the \emph{joining measure} of $\nu_1$
and $\nu_2$.
\end{cor}

This corollary will be used in combination with the following basic lemma. 

\begin{lem}\label{diagonalesansatome} 
Let $m_1$, $m_2$  be probability measures on a topological space $X$ 
and let $\Delta_X:=\{ (x,x)\mid x\in X\}$ be the diagonal of $X$.

If $m_1\otimes m_2 (\Delta_X) = 1$, 
then $m_1$ and $m_2$ are identical Dirac measures. 
\end{lem}

\begin{proof}
By assumption,  we have 
$m_1\otimes m_2 (\Delta_X) = \int_X m_1(\{x\}) \dd m_2 (x) = 1.$
Hence, for $m_2$-almost every $x\in X$, we have $m_1(\{x\})=1$, 
which implies that measures $m_1$ and $m_2$ are identical Dirac measures. 
\end{proof}

%42
\subsection{No Stationary Measures on $X_{V,W}$}

\bq
In this section, we again use
the same notations and assumptions as in Theorem \ref{thmeq}. 
We will prove that 
the space $X_{V,W}$ supports no $\mu$-stationary measures. 
\eq

Recall that $W\subset V$ are real vector spaces,
$G$ is a Zariski connected algebraic subgroup of $\Gl$
preserving $W$ and satisfying $(H1)$, $(H2)$, $(H3)$,
%We identify the quotient $W'=V/W$ with   
%a complementary subspace $W_s$ of $W$ in $V$.
Also recall that $\mu$ is a Zariski dense probability measure on $G$
with compact support, that $\lun$ and $\lpun$ 
are the first Lyapunov exponents of $\mu$ in $W$ and in $W':=V/W$,
and that we are studying the associated random walk on
the $G$-space  $X_{V,W}:=\Pj[V]\smallsetminus\Pj[W]$.

\begin{prop}\label{l1lp1nonrecloiprop} Same notations and assumptions 
as in Theorem \ref{thmeq}.\\
If $\lun \geq \lpun$, then the \ma on $X_{V,W}$ 
is nowhere recurrent in law,
and there exists no $\mu$-stationary probability measure on $X_{V,W}$. 
\end{prop}

\begin{proof} By Lemma \ref{urloimesinv} the first assertion follows from the second one.
This second assertion is a consequence of the following Lemmas 
\ref{liminfdnanfinie} and \ref{liminfdnaninfinie}.
\end{proof}

Let $B=G^{\N^*}$ and
$\beta=\mu^{\otimes\N^*}$. For $b=(b_1,b_2,\ldots)$ in $B$  
we write as in \eqref{eqnagcgdg}:
\begin{equation}
\label{eqnancndn}
b_1\cdots b_n=\begin{pmatrix}
a_n & c_n \\
0 & d_n
\end{pmatrix}.
\end{equation}

\begin{lem}\label{liminfdnanfinie}
Same notations and assumptions 
as in Theorem \ref{thmeq}.
If there exists a $\mu$-stationary probability measure
on $X_{V,W}$, 
then for $\beta$-almost every $b\in B$, 
we have
\begin{equation}
\label{eqnlimdan1}
\sup_{n\geq 1}\;  \n[a_n]/\n[d_n]\; <\; \infty.
\end{equation}
\end{lem}
The proof of Lemma \ref{liminfdnanfinie} will be given in Section \ref{secaccpoi}.
It relies on the properties of the limit probability measures $\nu_b$.

\begin{lem}\label{liminfdnaninfinie}
Same notations and assumptions 
as in Theorem \ref{thmeq}.
If $\lun\geq \lpun$, 
then for $\beta$-almost every $b\in B$, 
one has 
\begin{equation}\label{eqnlimdan2}
\sup_{n\geq 1}\;  \n[a_n]/\n[d_n]\; =\; \infty.
\end{equation}
\end{lem}
The proof of Lemma \ref{liminfdnaninfinie} will be given 
in Section \ref{secprocar}.
It relies on  the law of large numbers 
and on the law of the iterated logarithm for the random variables
$\log\| a_n\|\!-\!\log\|d_n\|$.

%43
\subsection{Using the Limit Measures}
\label{secaccpoi}

\bq
The aim of this section is to prove Lemma \ref{liminfdnanfinie}. 
\eq

We will need the following analog of \cite[Prop. 3.7]{Livre}
for a non-irreducible action. 
%It ensures that some stationary measure 
%does not mass proper subspaces.

\begin{lem}\label{pasdemasseausousespace}
Same notations and assumptions 
as in Theorem \ref{thmeq}.
Let  $\nu$ be 
a $\mu$-stationary probability measure
on $\Pj[V]$ such that $\nu(\Pj[W])=0$.
Then 
for every proper subspace $U$ of $V$, 
we have
$
\nu(\Pj[U])=0.
$
\end{lem}

\begin{proof}
Assume there exists a proper subspace $U$ of $V$ 
such that $\nu(\Pj[U])>0$. Let $r_0$ be the minimal dimension
of such a subspace $U$.
If $U_1$ and $U_2$ are two distinct vector subspaces of dimension $r_0$, 
one has the equality
%\[\nu(\Pj[U_1]\cup\Pj[U_2])=\nu(\Pj[U_1])+\nu(\Pj[U_2])+\nu(\Pj[U_1\cap U_2]);\]
%because the dimension $r$ of $U_1$ and $U_2$ is minimal, 
%$\nu(\Pj[U_1\cap U_2])$ is zero. 
%Thus, the following holds:
\[\nu(\Pj[U_1]\cup\Pj[U_2])=\nu(\Pj[U_1])+\nu(\Pj[U_2]).\] 
Let $\alpha: = \sup \{\nu(\Pj[U]) \,|\, U\subset V,\, \dim U = r_0\}>0\;\;$
and consider the set 
\[F= \{ U\subset V\,|\,\nu(\Pj[U])=\alpha,\,  \dim U = r_0\}.\]
This set is finite and non-empty. By $\mu$-stationarity of $\nu$,
for $\mu$-almost every $g\in G$, we have $g^{-1}F=F$. 
Therefore, since $\mu$ is Zariski dense in $G$, 
this set $F$ is $G$-invariant.
Since $G$ is Zariski connected, all the subspaces $U$ belonging to $F$ are $G$-invariant.
But by $(H1)$, $(H2)$ and $(H3)$, the only proper 
$G$-invariant subspace of $V$ is $W$.
This is contradictory since, by assumption, we have $\nu(\Pj[W])=0$.
\end{proof}

\begin{proof}[Proof of Lemma \ref{liminfdnanfinie}]
We assume also that
there exists a $\mu$-stationary probability measure $\nu$
on  $X_{V,W}$.
In order to prove \eqref{eqnlimdan1}, it is enough to check that
for $\beta$-almost every $b\in B$, 
for all accumulation points $\pi$ in ${\rm End}(V)$ of the sequence 
$p_n:=\frac{b_1\cdots b_n}{\n[b_1\cdots b_n]}$,
the image of $\pi$ is not included in $W$ :
\begin{equation}
\label{eqnimpinw}
\Ima \pi \not\subset W.
\end{equation}
Lemma \ref{pasdemasseausousespace}
shows that $\nu(\Pj[\Ker \pi])=0$, 
hence the image probability measure $\pi_*\nu$ is well-defined
and the sequence ${p_n}_*\nu$ weakly converges to $\pi_*\nu$. 
By Lemma \ref{nub} this sequence ${p_n}_*\nu$ also weakly converges to $\nu_b$, 
and therefore we have
\[\pi_*\nu=\nu_b.\]
Therefore, for $\beta$-almost all $b$ in $B$, one has, for all accumulation point $\pi$,
\[\nu_b(\Pj[\Ima\pi]) =1.\]
Since $\nu(\Pj[W])=0$, one also has, for $\beta$-almost all $b$ in $B$,
\[\nu_b(\Pj[W])=0, \]
and hence the images $\Ima\pi$ are not contained in $W$. 
This proves \eqref{eqnimpinw}.
\end{proof}

%44
\subsection{Using the Cartan Projection}
\label{secprocar}
 
\bq
The aim of this section is to prove Lemma \ref{liminfdnaninfinie}. 
\eq 

Let $\rho$ be the natural projection
\begin{equation}
\label{eqnrhoggl}
\rho\; :\;
G\longrightarrow \mathrm{GL}(W)\times\mathrm{GL}(W')
\;\; ;\;\; 
\begin{pmatrix}
a & c \\
0 & d
\end{pmatrix}
\longmapsto (a,d)\, .
\end{equation}
The image group
$
\overline{G}:=\rho(G)
$
is a reductive subgroup
of $\mathrm{GL}(W)\times\mathrm{GL}(W')$.
The image measure
$
\overline{\mu}:=\rho_*\mu
$
is a Zariski dense probability measure on $\overline{G}$. 

The proof of Lemma \ref{liminfdnaninfinie} will use the notations of 
Appendix \ref{seclimlaw}
with the reductive group $\overline{G}$
and its probability measure $\overline{\mu}$. 
In particular, 
$\Lie[g]$ is the Lie algebra of $\overline{G}$,
$\Lie[a]$ is the Lie algebra of a maximal split torus of  $\overline{G}$,
$\kappa$ is the Cartan projection,
$\sigma_{\overline{\mu}}$ is the Lyapunov vector, 
$\Phi_{\overline{\mu}}$ is the covariance $2$-tensor,
$\Lie[a]_{\overline{\mu}}$ is its linear span, and
$K_{\overline{\mu}}$ is the unit ball of $\Lie[a]_{\overline{\mu}}$.

We will also use the following two lemmas. 
We set $r=\dim W$ and $r'=\dim W'$.

\begin{lem}
\label{lemhigdis}
The highest weights $\chi$ and $\chi'$ of the representations 
of $\overline{G}$ in $W$ and $W'$ are distinct. 
\end{lem}

\begin{proof}
Since $\overline{G}$
is Zariski connected, Condition $(H1)$ tells us that $W$ and $W'$
are irreducible representations of $\Lie[g]$ and that 
their highest weight spaces
are one dimensional. Condition $(H2)$ tells us that these
representations of $\Lie[g]$ are not equivalent.
Therefore as in \cite[Chap 8.6.3]{BourGALie78}, the highest weights 
$\chi$ and $\chi'$ must be distinct.
\end{proof}

\begin{lem}\label{centreoverG}
The center $Z$ of $\overline{G}$ is equal to $Z=\{(
\alpha I_r ,\beta I_{r'})\in \overline{G}\,|\, \alpha,\,\beta \in\R^*\}.$
\end{lem}

\begin{proof} By Schur's lemma,
the commutant of $\overline{G}$ in ${\rm End}(W)$
is a division algebra.
Since the representation of $\overline{G}$ in $W$ is proximal,
this commutant is the field $\R$ of scalar matrices.
Therefore $Z$ acts on $W$ (and also on $W'$) by scalar matrices. 
\end{proof} 

\begin{proof}[Proof of Lemma \ref{liminfdnaninfinie}]
Fix norms on $W$ and $W'$ as in Lemma \ref{representationgeometriquecartaniwas},
so that, for any element $g=(a,d)$ in $\overline{G}$
with $a\in\mathrm{GL}(W)$, $d\in\mathrm{GL}(W')$, one has 
\begin{equation}
\label{eqnlogadg}
\log \|a\|\; =\; \chi(\kappa(g))
\;\; {\rm and}\;\;
\log \|d\|\; =\; \chi'(\kappa(g)) 
\end{equation}
 
In particular, the first Lyapunov exponents in $W$ and $W'$ are given by
\begin{equation*}
\lun\; =\; \chi(\sigma_{\overline{\mu}})
\;\; {\rm and}\;\;
\lpun\; =\;\chi'(\sigma_{\overline{\mu}}). 
\end{equation*}
Let $\overline{B}=\overline{G}^{\N^*}$ and $\overline{\beta}=\overline{\mu}^{\otimes \N^*}$. 
For $b=(b_1,b_2,\ldots)\in\overline{B}$, we write
$b_1\cdots b_n=(a_n,\,d_n).$
We distinguish three cases:
\vspace{1em}

\noindent {\bf First case :} $\lun > \lpun$.\;\;
In this case one has $(\chi-\chi')(\sigma_{\overline{\mu}}) >0$. 
According to \eqref{eqnlogadg} and the Law of Large Numbers \ref{corinv}, 
for $\overline{\beta}$-almost every $b\in \overline{B}$, 
we have
\begin{equation*}
\lim_{n\rightarrow\infty}\log (\n[a_n]/\n[d_n])
\; =\;
\lim_{n\rightarrow\infty}(\chi-\chi')(\kappa(b_1\cdots b_n))
\; =\; 
\infty . 
\end{equation*}

\noindent {\bf Second case :} $\lun=\lpun$ and $ (\chi-\chi')(\Lie[a]_{\overline{\mu}})\neq 0$.\;
In this case, one has $(\chi-\chi')(\sigma_{\overline{\mu}})=0$
and there exists $x$ in the unit ball $K_{\overline{\mu}}$ of 
$\Lie[a]_{\overline{\mu}}$
such that $(\chi-\chi')(x)>0.$
According to the Law of the Iterated Logarithm \ref{corinv}, 
for $\overline{\beta}$-almost every $b\in\overline{B}$, 
there exists an increasing sequence  of integers $n_i$ such that
\[\lim_{i\rightarrow\infty}\frac{\kappa(b_1\cdots b_{n_i})-n_i\,\sigma_{\widetilde{\mu}}}{\sqrt{2n_i\log \log n_i}}\; =\; x,\]
and therefore such that
\begin{equation*}
\lim_{i\rightarrow\infty}\log (\n[a_{n_i}]/\n[d_{n_i}])
\; =\;
\lim_{i\rightarrow\infty}(\chi-\chi')(\kappa(b_1\cdots b_{n_i}))
\; =\; 
\infty . 
\end{equation*}

\noindent {\bf Third case :} $\lun=\lpun$ and $(\chi-\chi')(\Lie[a]_{\overline{\mu}})= 0$.\;
Let 
$$
S:=\{(a,d)\in \overline{G}\mid |\det a|=|\det d|=1\}.
$$
Since the group $\overline{G}$ is reductive,
by Lemma \ref{centreoverG}, the subgroup $S$ is semisimple.
Let $\Lie[s]$ be the Lie algebra of $S$.
By \cite[Thm 13.19]{Livre}, 
we have $\Lie[a]\cap \Lie[s]\subset \Lie[a]_{\overline{\mu}}$,
and thus also
\begin{equation}
\label{eqnchias}
(\chi-\chi')(\Lie[a]\cap \Lie[s])=0.
\end{equation}
We introduce the group morphism $\delta$ defined by:
$$
\delta\; :\; \overline{G}\longrightarrow \R
\;\; ;\;\;
(a,d)\longmapsto \frac{1}{r}\log|\det a| -\frac{1}{r'}\log|\det d|.
$$ 
For every $g=(a,d)$ in $\overline{G}$, 
we can write $g=sz$ with $s\in S$ and $z\in Z$.
Using Equations \eqref{eqnlogadg}, \eqref{eqnchias}
and the equality $\kappa(g)=\kappa(s)+\kappa(z)$,
we compute, 
\begin{equation}
\label{eqnchichi}
\log\left(\n[a]/\n[d]\right)
\; =\;
(\chi-\chi')(\kappa(g))
\; =\;
(\chi-\chi')(\kappa(z))
\; =\;
\delta(z)
\; =\;
\delta(g).
\end{equation}
We want to describe the behavior of  the random variable 
$
T_n=\log \left(\n[a_n]/\n[d_n]\right)
$
on $\overline{B}$ where as above $(a_n,\,d_n)= b_1\cdots b_n$.
Using Equation \eqref{eqnchichi}, we see that 
$$
T_n =\delta(b_1\cdots b_n)= \delta(b_1)+\cdots +\delta(b_n)
$$
is the sum of $n$ real-valued independent and identically distributed
random variables $\delta(b_i)$. 
Note that the law of the variable $\delta(b_1)$ has compact support.
Since $\lun=\lpun$, we have $\mathbb{E}(\delta(b_1))=\frac1n\mathbb{E}(T_n)\xrightarrow[\nti]{} 0$.
Thus the variable $\delta(b_1)$ is centered. 
%by the law of large number \ref{corinv}, 
%the sequence $\frac1n T_n$ converges almost-surely to $0$.
%According to the classical Law of Large Numbers 
%(cf. e.g. \cite[Thm 3.30]{Brei}), 
%this implies that the variable $\delta(b_1)$ is centered.

If this random variable $\delta(b_1)$ were almost surely $0$, 
it would mean that for $\overline{\mu}$-almost every 
$g=(a,\,d)\in \overline{G}$, we have $\delta(g)=0$. 
Since $\overline{\mu}$ is Zariski dense in $\overline{G}$,
this would imply $\delta(G)=0$, or, equivalently,
\begin{equation}
\label{eqndegnul}
(\chi-\chi')(\Lie[z])=0,
\end{equation}
where $\Lie[z]\subset\Lie[a]$ is the Lie algebra of $Z$. 
Equalities \eqref{eqnchias} and \eqref{eqndegnul}
would tell us that the highest weights $\chi$ and $\chi'$ were equal.
This would contradict Lemma \ref{lemhigdis}.  

Therefore this centered variable $\delta(b_1)$ is not almost surely $0$.
Thus the classical recurrence properties of real random walks 
(cf e.g. \cite[Thm 3.38]{Brei})
tell us that $\sup_{n\geq 1}T_n=\infty$ almost surely. 
\vspace{1em}

In each of these three cases, we have checked  \eqref{eqnlimdan2}.
\end{proof}

%50
\section{Uniqueness of the Stationary Measure }\label{thmequnicite}
\label{secunista}

\bq
The main aim of this chapter is to prove the uniqueness of the stationary measure on $X_{V,W}$ (Proposition \ref{unicitedirac}).
\eq

%51
\subsection{No Stationary Measures on $Y_{V,W}$}

\bq
The proof of uniqueness will rely on the following  Lemma 
\ref{lemyvw}.
\eq

We keep the notations and assumptions 
of Theorem \ref{thmeq}.
Let $p$ be the projection
$$ 
p\; :\; X_{V,W}\longrightarrow \Pj[W']
\;\; ;\;\; [v]\longmapsto [v+W]
$$ 
and let $Y_{V,W}$
be the $G$-invariant subvariety of $X_{V,W}^2$
\[Y_{V,W}:= \{(x,\,x')\in X_{V,W}^2\,|\,p(x)=p(x'),\, x\neq x'\}.\] 

\begin{lem}\label{lemyvw}
Same notations and assumptions 
as in Theorem \ref{thmeq}.\\
There is no $\mu$-stationary probability measure $\widetilde{\nu}$ on $Y_{V,W}$. 
\end{lem}

\begin{proof}
Suppose that such a measure $\widetilde{\nu}$ does exist.  
Consider again the natural projection 
$\rho : G\longrightarrow \mathrm{GL}(W)\times\mathrm{GL}(W')$
introduced in \eqref{eqnrhoggl}.
Let
$
\overline{G}:=\rho(G)
$
be the image of $G$ by $\rho$, a reductive subgroup
of $\mathrm{GL}(W)\times\mathrm{GL}(W')$,
and let
$
\overline{\mu}:=\rho_*\mu
$
be the image of $\mu$ by $\rho$, 
a Zariski dense probability measure on $\overline{G}$. 
Now consider the map
$$
f\; :\;
Y_{V,W}\longrightarrow \overline{Y}
\;\; ;\;\; 
( \lbrack w_1,\, w' \rbrack ,\lbrack w_2,\, w' \rbrack)
\longmapsto
\lbrack w_1-w_2,\, w' \rbrack ,
$$
where  $\overline{Y}:=\Pj[W\oplus W']\smallsetminus (\Pj[W]\cup\Pj[W'])$. 
Let $\overline{\nu}=f_*\widetilde{\nu}$ 
be the probability measure on $\overline{Y}$ 
that is the image of $\widetilde{\nu}$ by $\rho$. 
Since the map $f$ is equivariant, the probability measure 
$\overline{\nu}$ is $\overline{\mu}$-stationary. 
According to Proposition \ref{BQorbitescompactesmeasures} 
such a measure $\overline{\nu}$ is supported by a compact
$\overline{G}$-orbit in $\overline{Y}$. This contradicts the following 
Lemma \ref{pasdorbitescompactesdansX}.
\end{proof}

\begin{lem}\label{pasdorbitescompactesdansX} 
There are no compact $\overline{G}$-orbits in $\overline{Y}$.
\end{lem}

\begin{proof}
Such a compact  orbit would be of the form $\overline{G}/\overline{H}$, 
where $\overline{H}$ is an algebraic subgroup of $\overline{G}$ 
containing a conjugate of the group $\overline{A}\,\overline{N}$
with $\overline{A}$ a maximal split subtorus of $\overline{G}$
and $\overline{N}$ a maximal unipotent subgroup normalized by $\overline{A}$. 
Since $W$ and $W'$ are proximal irreducible representations of $\overline{G}$, 
there is only one $\overline{N}$-invariant line $\R v$ 
in $W$ and one $\R v'$ in $W'$.
Hence the $\overline{N}$-invariant lines 
in $W\oplus W'$ are included in the plane $\R v\oplus\R v'$. 
Since, by Lemma \ref{lemhigdis},
the highest weights $\chi$ and $\chi'$ of $W$ and $W'$
are distinct, the lines $\R v$ and $\R v'$
are the only $\overline{A}$-invariant lines in  $\R v \oplus \R v'$. 
Therefore, a compact $\overline{G}$-orbit in $\Pj[W\oplus W']$ 
is contained in $\Pj[W]\cup\Pj[W']$.
\end{proof}

%52
\subsection{Proof of Uniqueness}
\bq
We can now show the uniqueness of 
the $\mu$-stationary probability measure $\nu$ on $X_{V,W}$. 
The same proof will tell us that its 
limit probability measures $\nu_b$
are Dirac measures. 
\eq

\begin{prop}
\label{unicitedirac}
Same notations and assumptions 
as in Theorem \ref{thmeq}.\\
If $\lun<\lpun$, the $\mu$-stationary probability measure 
$\nu$ on $X_{V,W}$ is unique.\\
Moreover, the limit measures $\nu_b$ are $\beta$-almost surely Dirac measures. 
\end{prop}

\begin{proof}[Proof of Proposition \ref{unicitedirac}] 
Let $\nu_1$ and $\nu_2$ be two $\mu$-stationary probability  measures on $X_{V,W}$.
By Corollary \ref{corjoista} the joining measure $\nu_1\boxtimes\nu_2$ 
on $X_{V,W}^2$ is $\mu$-stationary.

Let us show that its support is contained in the subvariety
\[Z_{V,W}:= \{(x,\,x')\in X_{V,W}^2\,| \,p(x)=p(x')\},\]
where $p:X_{V,W}\rightarrow \Pj[W']$ is again the canonical projection. 
Since the action of $G$ on $W'$ is irreducible and proximal, 
there exists a unique $\mu$-stationary measure $\nu'_0$
on $\Pj[W']$ called the {\it Furstenberg measure}. Its limit probability 
measures $\nu'_{0,\,b}$ 
are $\beta$-almost surely Dirac measures
$\delta_{\xi_b}$
for some $\xi_b\in\Pj[W']$.
See \cite[Prop. 3.7]{Livre} for more detail on the Furstenberg measure.
Since $\nu'_0$ is unique,
we have the equalities
$$
p_*\nu_1=p_*\nu_2=\nu'_0 .
$$
Therefore, for $\beta$-almost every $b\in B$, we have
$$
p_*\nu_{1,b}=p_*\nu_{2,b}=\delta_{\xi_b},
$$ 
and hence 
$$
\nu_{1,b}\otimes \nu_{2,b}(Z_{V,W})=1.
$$
By the very definition \eqref{eqnjoista} of the joining measure, integrating this equality gives
$$
\nu_1\boxtimes\nu_2(Z_{V,W})=1.
$$
By definition, this set $Z_{V,W}$ is the union $Z_{V,W}=Y_{V,W}\cup \Delta_{X_{V,W}}$.
By Lemma \ref{lemyvw},
the $G$-variety $Y_{V,W}$ does not support $\mu$-stationary measures.
Therefore the  joining measure
$\nu_1\boxtimes\nu_2$
is supported on the diagonal $\Delta_{X_{V,W}}$.
Hence, for $\beta$ almost every $b$ in $B$, the measure
$\nu_{1,b}\otimes\nu_{2,b}$ is also supported on the diagonal:
$
\nu_{1,b}\otimes\nu_{2,b}(\Delta_{X_{V,W}})=1.
$
Therefore, by Lemma \ref{diagonalesansatome}, 
the limit probability measures $\nu_{1,b}$ and $\nu_{2,b}$ are 
both equal to the same Dirac measures. Hence, by Lemma \ref{nub}, one has $\nu_1=\nu_2$.
\end{proof}

%This ends the proof of Proposition \ref{unicitedirac}.

%53
\subsection{Limit of Means of Transition Probabilities}
\label{secconmea}

\bq
In this section we prove that the sequence of 
means of the transition probabilities $\mu^{*n}*\delta_x$ 
on $X_{V,W}$ always has a limit.
\eq

\begin{cor}\label{corthmeqlunlpun}
Same notations and assumptions 
as in Theorem. \ref{thmeq}. 
Let $x\in X_{V,W}$.\\
a) When $\lun\geq\lpun$,
one has the weak convergence
$\frac{1}{n}\sum_{j=1}^n \mu^{*j}* \delta_x\xrightarrow[\nti]{} 0.$\\
b) When $\lun<\lpun$,
one has the weak convergence
$\frac{1}{n}\sum_{j=1}^n \mu^{*j}* \delta_x\xrightarrow[\nti]{} \nu.$
\end{cor}

\begin{proof}[Proof of Corollary \ref{corthmeqlunlpun}]
Every accumulation point
of the sequence of probability measures
$\frac{1}{n}\sum_{j=1}^n \mu^{*j}* \delta_x$
is a $\mu$-stationary  finite measure. 

When $\lun\geq\lpun$, by Proposition \ref{l1lp1nonrecloiprop},
such a measure is necessarily $0$. 

When $\lun<\lpun$, by Corollary \ref{l1lp1urloi},
the corresponding Markov chain is recurrent in law; 
hence, no mass is lost
and the accumulation points are thus $\mu$-stationary probability measures.
By Proposition \ref{unicitedirac}, there is only one such measure. 
\end{proof}

%This concludes the proof of Theorem \ref{thmeq}. 

%A0
\appendix

\section{Limit Laws on Reductive Groups}
\label{seclimlaw}

\bq
In this appendix, we recall some facts about random walks on reductive groups, 
which are mainly detailed in \cite{Livre}. 
\eq

%A1
\subsection{Cartan Decomposition}

Let $G$ be a Zariski connected real algebraic reductive group. 
Let $A$ be a maximal split subtorus of $G$,
$\Lie[a]$ be the Lie algebra  of $A$,
$\Lie[a]_+\subset\Lie[a]$ be a Weyl chamber
and $A_+=\exp \Lie[a]_+$.
There exists a maximal compact subgroup $K$ of $G$ 
such that $G$ has a \emph{Cartan decomposition}
$G=KA_+ K.$
The \emph{Cartan projection} of $G$ is the unique map 
$\kappa : G \rightarrow \Lie[a]_+$
such that, for all $g\in G$,
\[g\in K \text{exp}(\kappa(g)) K.\]

The Cartan projection is useful because of the following Lemma.

\begin{lem}\label{representationgeometriquecartaniwas}
(\cite[Lem. 6.33]{Livre}) 
Let $G$ be a Zariski connected real algebraic reductive group,
$\rho$ be an irreducible algebraic representation of $G$ 
in a real vector space $W$ and $\chi\in \Lie[a]^*$
be the \emph{highest weight of $W$}.
There exists a norm  on $W$  such that, 
for all $g\in G$, 
\begin{equation}
\label{eqnchikap}
\chi (\kappa(g))=\log (\n[\rho(g)]).
\end{equation}
\end{lem}

%A2
\subsection{Limit Laws}

Let $\mu$ be  a Zariski dense probability measure with compact support
on $G$.
Let $B=G^{\N^*}$ and $\beta=\mu^{\otimes\N^*}$. 
We now recall two limit laws for the Cartan projection.

Define the \emph{Lyapunov vector} (see \cite[Thm 10.9]{Livre})
$\sigma_\mu\in \Lie[a]_+$ as the limit 
\begin{equation}
\label{eqnsimu}
\sigma_{\mu}:=\lim_{n\rightarrow\infty}
\frac{1}{n}\int_G\kappa(g)\dd\mu^{*n}(g).
\end{equation}

\begin{thm}\label{thmLGN}(Law of Large Numbers, \cite[Thm 10.9]{Livre})
Let $G$ be a Zariski connected real reductive group 
and $\mu$ a Zariski dense probability measure with compact support on $G$. 
For $\beta$-almost every $b\in B$, 
we have the convergence
\[\tfrac{1}{n}\,\kappa(b_n\cdots b_1)\xrightarrow[\nti]{}\sm.\]
\end{thm}

Define the \emph{covariance 2-tensor} (see \cite[Prop.14.18]{Livre})
$\Phi_{\mu}\in S^2\Lie[a]$ 
as the limit
\begin{equation}
\label{eqnphimu}
\Phi_{\mu}:= \lim_{n\rightarrow\infty}\frac{1}{n}\int_G (\kappa(g)-n\sigma_{\mu})^{\otimes 2} \dd\mu^{*n}(g).
\end{equation}
$\Phi_{\mu}$ is a symmetric $2$-tensor on $\Lie[a]$.
We denote by $\Lie[a]_{\mu}\subset\Lie[a]$ the 
\emph{linear span of $\Phi_{\mu}$} which  
is the smallest subspace $\Lie[a]_{\mu}$ of $\Lie[a]$
such that we have $\Phi_{\mu}\in S^2\Lie[a]_{\mu}$. 
We can see $\Phi_{\mu}$ as an inner product over $\Lie[a]_{\mu}$.  
We then denote by  $K_{\mu}$ the closed unit ball of $\Lie[a]_{\mu}$
for the metric corresponding to $\Phi_{\mu}$. 
When $G$ is semisimple, the covariance $2$-tensor $\Phi_{\mu}$ is 
\emph{non degenerate} 
i.e. one has $\Lie[a]_{\mu}=\Lie[a]$.

\begin{thm}\label{thmLIL}
(Law of the Iterated Logarithm, \cite[Thm 13.17]{Livre})
Let $G$ be a Zariski connected real reductive  group and
$\mu$ a Zariski dense probability measure with compact support on $G$. 
Then, for $\beta$-almost every $b\in B$, 
the set of accumulation points of  the sequence
\[(\frac{\kappa(b_n\cdots b_1)-n\sigma_{\mu}}{\sqrt{2n\log\log n}})\sui
\;\;\; \mbox{\it is exactly $K_{\mu}$.}\]
\end{thm} 

%A3
\subsection{Opposition Involution}

In Chapter \ref{l1lp1nrecloi} we need a variation of Theorems 
\ref{thmLGN} and \ref{thmLIL} where the order in the product of $b_i$'s is inverted.

\begin{cor}\label{corinv}
Let $G$ be a Zariski connected real reductive  group
and $\mu$ a Zariski dense probability measure with compact support on $G$.\\
a) For $\beta$-almost every $b\in B$, 
we have the convergence
\[\tfrac{1}{n}\,\kappa(b_1\cdots b_n)\xrightarrow[\nti]{}\sm.\]
b) For $\beta$-almost every $b\in B$, 
the set of accumulation points of the sequence
\[(\frac{\kappa(b_1\cdots b_n)-n\sigma_{\mu}}{\sqrt{2n\log\log n}})\sui
\;\;\; \mbox{\it is exactly $K_{\mu}$.}\]
\end{cor}
 
The proof will use the probability measure $\check{\mu}$ on $G$ 
which is the image of $\mu$ by 
the map $g\mapsto g^{-1}$.
Recall that there exists a linear map 
$\iota:\Lie[a]\rightarrow\Lie[a]$ 
called the \emph{opposition involution} (see \cite[\S 8.2]{Livre})
such that, for all $g\in G$, we have 
\begin{equation}
\label{eqnkapg}
\kappa(g^{-1})=\iota(\kappa(g)).
\end{equation}

\begin{lem}\label{iotasigmamu}
The Lyapunov vector $\sigma_{\check{\mu}}$ of $\check{\mu}$, 
its covariance $2$-tensor $\Phi_{\check{\mu}}$
and the closed unit ball $K_{\check{\mu}}$  of 
the linear span $\Lie[a]_{\check{\mu}}$ of $\Phi_{\check{\mu}}$ are equal to:
\begin{equation*}
\sigma_{\check{\mu}}\; =\; \iota(\sm)
\;\; ,\;\;
\Phi_{\check{\mu}} \; =\; \iota(\Phi_{\mu}) 
\;\; ,\;\;
K_{\check{\mu}}\; =\;\iota(K_{\mu}).
\end{equation*}
\end{lem}

\begin{proof} These identities follow from
\eqref{eqnsimu}, \eqref{eqnphimu} and \eqref{eqnkapg}.
\end{proof}

\begin{proof}[Proof of Corollary \ref{corinv}]
We apply the Law of Large Numbers \ref{thmLGN} 
to the probability measure $\check{\mu}$. 
For $\beta$-almost every $b\in B$, 
the sequence $\frac{1}{n}\kappa(b_n^{-1}\cdots b_1^{-1})$ converges to 
$\sigma_{\check{\mu}}$. 
Applying the opposition involution $\iota$, we find that,
for $\beta$-almost every $b\in B$, 
the sequence $\frac{1}{n}\kappa(b_1\cdots b_n)$ converges to
$\iota(\sigma_{\check{\mu}})$ which is equal to $\sm$ by Lemma \ref{iotasigmamu}. 

In the same manner,  Theorem \ref{thmLIL} tells us
that the set of accumulation points of the sequence 
$\frac{\kappa(b_1\cdots b_n)-n\sigma_{\mu}}{\sqrt{2n\log\log n}}$
is exactly $\iota(K_{\check{\mu}})$ which is equal to 
$K_{\mu}$ by Lemma \ref{iotasigmamu}. 
\end{proof}

%A4
\subsection{Stationary Measures on Projective Spaces}

In Chapter \ref{secunista}, we use the 
classification  of stationary measures in \cite[Thm 1.7]{Art}:

\begin{prop}\label{BQorbitescompactesmeasures}
Let $V$ be a real vector space, $G\subset \Gl$ be a reductive algebraic 
subgroup of $\Gl$ and
$\mu$ be a Zariski dense probability measure on $G$. 

Then every $\mu$-stationary probability measure $\nu$  on $\Pj[V]$ 
is supported by a compact $G$-orbit.

Conversely every compact $G$-orbit in $\Pj[V]$ supports a unique 
$\mu$-stationary probability measure $\nu$.
\end{prop}

%\bibliography{recbib}{}
%\bibliographystyle{abbrv}

\small{
\bibliography{recbib}{}
\bibliographystyle{abbrv}
}
\vspace{1em}

{\small\noindent 
Yves Benoist : yves.benoist@math.u-psud.fr \\
Caroline Bru\`ere : caroline.arvis@math.u-psud.fr\\
Universit\'e Paris-Sud, 
Orsay 91405 France}

\end{document}